\documentclass[a4paper,reqno,10.5pt, oneside]{amsart}

\usepackage{amssymb}
\usepackage{amstext}
\usepackage{amsmath}
\usepackage{amscd}
\usepackage{amsthm}
\usepackage{amsfonts}
\usepackage{enumerate}
\usepackage{graphicx}
\usepackage{latexsym}
\usepackage{mathrsfs}
\usepackage{mathtools}
\usepackage{here}

\usepackage{tikz}
\usepackage{tikz-cd}
\usetikzlibrary{arrows}
\usetikzlibrary{decorations.markings}
\usetikzlibrary{patterns,decorations.pathreplacing}

\makeatletter
\pgfdeclarepatternformonly[\LineSpace]{my north west lines}{\pgfqpoint{-1pt}{-1pt}}{\pgfqpoint{\LineSpace}{\LineSpace}}{\pgfqpoint{\LineSpace}{\LineSpace}}
{
    \pgfsetcolor{\tikz@pattern@color} 
    \pgfsetlinewidth{0.6pt}
    \pgfpathmoveto{\pgfqpoint{0pt}{\LineSpace}}
    \pgfpathlineto{\pgfqpoint{\LineSpace + 0.1pt}{- 0.1pt}}
    \pgfusepath{stroke}
}
\makeatother 
\newdimen\LineSpace
\tikzset{
    line space/.code={\LineSpace=#1},
    line space=8.5pt
}

\usepackage{hyperref}

\newtheorem{theorem}{Theorem}[section]

\newtheorem{corollary}[theorem]{Corollary}
\newtheorem{lemma}[theorem]{Lemma}
\newtheorem{proposition}[theorem]{Proposition}
\newtheorem{definition-proposition}[theorem]{Definition-Proposition}

\theoremstyle{definition}
\newtheorem{definition}[theorem]{Definition}
\newtheorem{example}[theorem]{Example}
\newtheorem{observation}[theorem]{Observation}

\newtheorem{settings}[theorem]{Settings}

\theoremstyle{remark}
\newtheorem{remark}[theorem]{Remark}

\makeatletter
  
  \@addtoreset{equation}{section}
\makeatother

\newcommand{\calA}{\mathcal{A}}

\newcommand{\calC}{\mathcal{C}}

\newcommand{\calU}{\mathcal{U}}
\newcommand{\calS}{\mathcal{S}}

\newcommand{\Sf}{\mathfrak{S}}

\newcommand{\fkm}{\mathfrak{m}}

\newcommand{\NN}{\mathbb{N}}
\newcommand{\ZZ}{\mathbb{Z}}

\newcommand{\RR}{\mathbb{R}}

\newcommand{\TT}{\mathbb{T}}
\newcommand{\kk}{\Bbbk}

\newcommand{\bfa}{\mathbf{a}}
\newcommand{\bfc}{\mathbf{c}}

\newcommand{\bfx}{\mathbf{x}}
\newcommand{\sfM}{\mathsf{M}}
\newcommand{\sfN}{\mathsf{N}}

\newcommand{\rank}{\operatorname{rank}}

\newcommand{\Hom}{\operatorname{Hom}}

\newcommand{\Cl}{\operatorname{Cl}}

\newcommand{\des}{\operatorname{des}}

\newcommand{\add}{\mathsf{add}}

\begin{document}

\title[Generalized $F$-signatures of Hibi rings]{Generalized $F$-signatures of Hibi rings}
\author[A Higashitani \and Y. Nakajima]{Akihiro Higashitani \and Yusuke Nakajima} 

\address[A Higashitani]{ Department of Pure and Applied Mathematics, Graduate School of Information Science and Technology, Osaka University, Suita, Osaka 565-0871, Japan}
\email{higashitani@ist.osaka-u.ac.jp}

\address[Y. Nakajima]{Department of Mathematics, Kyoto Sangyo University, Motoyama, Kamigamo, Kita-Ku, Kyoto, Japan, 603-8555}
\email{ynakaji@cc.kyoto-su.ac.jp}



\maketitle

\begin{abstract} 
The $F$-signature is a numerical invariant defined by the number of free direct summands in the Frobenius push-forward, 
and it measures singularities in positive characteristic. 
It can be generalized by focussing on the number of non-free direct summands. 
In this paper, we provide several methods to compute the (generalized) $F$-signature of a Hibi ring 
which is a special class of toric rings. 
In particular, we show that it can be computed by counting the elements in the symmetric group satisfying certain conditions. 
As an application, we also give the formula of the (generalized) $F$-signature for some Segre products of polynomial rings. 

\end{abstract}

\setcounter{tocdepth}{1}
\tableofcontents

\section{Introduction} 
\label{sec_intro}

In this paper, we study the (\emph{generalized}) \emph{$F$-signature}, 
which is a numerical invariant defined for a commutative algebra in positive characteristic. 
The $F$-signature was introduced in \cite{HL} inspired by the work \cite{SmVdB}, and it measures singularities as we will see in Theorem~\ref{prop_Fsig}. 
The $F$-signature has been computed for many classes of commutative algebras, but its value is still mysterious. 
The present paper focuses on the (generalized) $F$-signature for a special class of toric rings, called \emph{Hibi rings}. 

\subsection{Backgrounds}
Let $(A,\fkm)$ be a Noetherian local ring of prime characteristic $p>0$. 
We assume that $\kk\coloneqq A/\fkm$ is an algebraically closed field throughout this paper. 
In this case, we can define the Frobenius morphism $F:A\rightarrow A \ (a\mapsto a^p)$. 
For $e\in\NN$, we also define the $e$-times iterated Frobenius morphism $F^e:A\rightarrow A \ (a\mapsto a^{p^e})$. 
For any $A$-module $M$, we define ${}^eM$, which is the $A$-module given by restriction of scalars under $F^e$. 
That is, ${}^eM$ is just $M$ as an abelian group, and for $m\in M$ if we denote the corresponding element of ${}^eM$ by ${}^em$ then 
its $A$-module structure is defined by $a({}^em)\coloneqq {}^e(F^e(a)m)={}^e(a^{p^e}m)$ for $a\in A$. 
Sometimes, ${}^eM$ is denoted by $F^e_*M$, and called the \emph{Frobenius push-forward} of $M$. 
We say that $A$ is \emph{$F$-finite} if ${}^1A$ (equivalently every ${}^eA$) is a finitely generated $A$-module. 
In positive characteristic commutative algebra, we investigate the properties of $A$ through the structure of ${}^eA$ (or ${}^eM$). 
For example, it is known that an $F$-finite local ring $A$ is regular if and only if ${}^eA$ is a free $A$-module \cite{Kun}. 
Furthermore, various classes of rings in positive characteristic are defined via ${}^eA$. 

In the following, we assume that $A$ is an $F$-finite domain for simplicity. 
Since $A$ is reduced in particular, we can identify ${}^eA$ with $A^{1/p^e}$, 
which is the $A$-algebra consisting of $p^e$-th roots of $A$ inside an algebraic closure of its fraction field, 
by associating ${}^ea$ with $a^{1/p^e}$ for any $a\in A$. 
With this viewpoint, the $e$-times iterated Frobenius morphism $F^e$ is identified with the inclusion $A\hookrightarrow A^{1/p^e}$. 

\begin{remark}
\label{KS condition}
In this paper, we focus on an indecomposable decomposition of $A^{1/p^e}$. 
Since $\widehat{A}^{1/p^e}\cong \widehat{A}\otimes_AA^{1/p^e}$ holds where $\widehat{A}$ is the $\fkm$-adic completion of $A$, 
we can reduce our arguments to $\widehat{A}$-modules. 
Since $\widehat{A}$ satisfies the Krull-Schmidt condition, we will assume that $A$ satisfies this condition, 
and hence an indecomposable decomposition of $A^{1/p^e}$ is unique up to isomorphism. 
\end{remark}

For understanding the structure of $A^{1/p^e}$, we also introduce the notion of finite $F$-representation type.

\begin{definition} (see \cite{SmVdB,Yao})
\label{def_FFRT}
We say that $A$ has \emph{finite $F$-representation type} (= \emph{FFRT}) by $\calS$ if there is a finite set 
$\calS\coloneqq\{M_0, M_1, \cdots, M_n\}$ of isomorphism classes of indecomposable finitely generated $A$-modules such that 
for any $e\in\NN$, the $A$-module $A^{1/p^e}$ is isomorphic to a finite direct sum of these modules: 
\begin{equation}
\label{A_decomp}
A^{1/p^e}\cong M_0^{\oplus c_{0,e}}\oplus M_1^{\oplus c_{1,e}}\oplus\cdots\oplus M_n^{\oplus c_{n,e}}
\end{equation}
for some $c_{i,e}\ge 0$. 
Moreover, we say that a finite set $\calS=\{M_0, \cdots, M_n\}$ is the \emph{FFRT system} of $A$ if every $A$-module $M_i$ appears non-trivially in 
$A^{1/p^e}$ as a direct summand for some $e\in\NN$. 
\end{definition}

Let $A$ be a ring having FFRT by the FFRT system $\calS=\{M_0, M_1, \cdots, M_n\}$. 
In order to detect the asymptotic behavior of the decomposition (\ref{A_decomp}), we consider the limit 
\begin{equation}
\label{def_genFsig}
s(M_i, A)\coloneqq\lim_{e\rightarrow\infty}\frac{c_{i,e}}{p^{ed}} \quad ( i, j=0,1,\cdots,n ). 
\end{equation}
This numerical invariant was derived from the study in \cite{SmVdB}, and 
the existence of the limit is guaranteed under the assumption $A$ has FFRT (see \cite[Theorem~3.11]{Yao}). 
Furthermore, $s(M_i,A)$ is a positive rational number if $A$ has FFRT and is strongly $F$-regular (see \cite[Proposition~3.3.1]{SmVdB}). 
Here, we say that $A$ is \emph{strongly $F$-regular} if for every $0\neq c\in A$ there exists $e\in\NN$ such that 
$$
A\hookrightarrow A^{1/p^e} \overset{\cdot c^{1/p^e}}{\longrightarrow}A^{1/p^e} \quad(a\mapsto c^{1/p^e}a)
$$
splits as an $A$-linear morphism. 
We remark that if $M_i=A$ for some $i$ (e.g., the case where $A$ is strongly $F$-regular), 
then $s(A, A)$ is nothing else but the \emph{$F$-signature} of $A$ introduced in \cite{HL}, and denoted by $s(A)$. 
The $F$-signature can be defined in more general settings (see Theorem~\ref{prop_Fsig} below). 
Also, it had been revealed that the $F$-signature coincides with the minimal relative Hilbert-Kunz multiplicity introduced in \cite{WY}. 
We collect basic properties on the $F$-signature $s(A)=s(A, A)$ here. 

\begin{theorem}
\label{prop_Fsig}
Let $A$ be a $d$-dimensional $F$-finite (complete) local domain of prime characteristic $p>0$. 
We suppose that $A^{1/p^e}\cong A^{\oplus a_e}\oplus M_e$ where $M_e$ contains no free direct summand. 
Then, the $F$-signature $s(A)\coloneqq\displaystyle\lim_{e\rightarrow\infty}\frac{a_e}{p^{ed}}$ exists $($see \cite{Tuc}$)$ and satisfies the following properties. 
\begin{enumerate}[\rm (1)]
\item $0\le s(A)\le 1$. 
\item $A$ is regular if and only if $s(A)=1$ $($see \cite{HL,WY}$)$. 
\item $A$ is strongly $F$-regular if and only if $s(A)>0$ $($see \cite{AL}$)$. 
\end{enumerate}
\end{theorem}

Therefore, the $F$-signature measures the deviation from regularity, 
and we expect that a strongly $F$-regular ring having large $F$-signature would have better singularities. 
With these backgrounds, we call $s(M_i, A)$ given in (\ref{def_genFsig}) the \emph{generalized $F$-signature} of $M_i$ (with respect to $A$). 
This numerical invariant has been investigated in several literature e.g., \cite{Bru,HaN,HS}. 

\begin{remark}
\label{decomp_M}
We can also consider the $A$-module $M_j^{1/p^e}$ and its asymptotic behavior in a similar way. 
Then, we also define the \emph{generalized $F$-signature} of $M_i$ with respect to $M_j$ which is denoted by $s(M_i,M_j)$. 
It is known that this satisfies $s(M_i,M_j)=(\rank_AM_j)\cdot s(M_i,A)$ if $A$ is strongly $F$-regular and has FFRT (see \cite[Proof of Proposition~3.3.1.]
{SmVdB}). 
\end{remark}

\subsection{Summary of our results}
In this paper, we study the generalized $F$-signatures for a Hibi ring, which is a special class of toric rings defined by a partially ordered set (= poset). 
First, it is known that any toric ring is strongly $F$-regular and has FFRT, 
and hence for each module in the FFRT system the generalized $F$-signature exists and it is strictly positive. 
Moreover, the FFRT system of a toric ring can be characterized by a special class of divisorial ideals called \emph{conic} (see Definition~\ref{def_conic}). 
Thus, the generalized $F$-signature is defined for each conic divisorial ideal. 
If $R$ is a Hibi ring, then conic divisorial ideals are characterized in terms of the associated poset (see Theorem~\ref{conic_Hibi}). 
Furthermore, we can see that each conic divisorial ideal of a $d$-dimensional Hibi ring corresponds to a certain full-dimensional cell in $(-1,0]^d$ 
(see Observation~\ref{obs_fulldim_cell}). 
It is known that the generalized $F$-signature of a conic divisorial ideal can be obtained as the volume of the corresponding cell 
(see Corollary~\ref{Thm_Fsig_volume_Hibi}). 
The point which we want to discuss in this paper is that when we consider a $d$-dimensional Hibi ring, 
the generalized $F$-signature can also be obtained as the number of certain elements in the symmetric group $\mathfrak{S}_d$. 
We will explain this for the case of the Segre product of polynomial rings, which is an important example of Hibi rings (see Theorem~\ref{S_n}). 
The crucial point is that the full-dimensional cell corresponding to a conic divisorial ideal is considered as an \emph{alcoved polytope}. 
This polytope was introduced by Lam and Postnikov \cite{LP}, 
and they showed that the volume of this polytope can be given in terms of a symmetric group. 
Thus, we have the followings. 

\begin{theorem}[{see Theorem~\ref{S_n}}] 
Let $R$ be the Segre product of polynomial rings with $\dim R=d$, and let $M$ be a conic divisorial ideal of $R$. 
Then, the generalized $F$-signature $s(M,R)$ is equal to $\Gamma_M(\mathfrak{S}_d)/d!$, 
where $\Gamma_M(\mathfrak{S}_d)$ is the number of elements in $\mathfrak{S}_d$ satisfying certain conditions that depend on $M$. 
\end{theorem}

By using this description, we can obtain 
the formulas of the generalized $F$-signatures for some classes of the Segre product of polynomial rings 
(see Proposition~\ref{compute_Fsig1} and \ref{compute_Fsig2}). 
Although, we mainly discuss the case of the Segre product of polynomial rings, our arguments are also valid for any Hibi ring 
(see Section~\ref{sec_Fsig_remark}).

\subsection{Plan of the paper} 

The content of this paper is the following. In Section~\ref{sec_pre_Fsig}, 
we prepare some results concerning the generalized $F$-signatures in a general setting. 
Then, in Section~\ref{sec_conic}, we prepare notations concerning toric rings, and review some known results. 
In Section~\ref{sec_Fsig_Hibi}, we introduce a Hibi ring. 
Then, we give the precise description of conic divisorial ideals of a Hibi ring using the associated poset, 
and explain that the generalized $F$-signature can be obtained as the volume of a certain polytope. 
In Section~\ref{sec_Fsig_Segre}, we give another description of the generalized $F$-signature in terms of 
a symmetric group for the Segre product of polynomial rings. 
Using such a description, in Section~\ref{sec_Fsig_formula}, we give some formulas of the generalized $F$-signatures. 
In Section~\ref{sec_Fsig_remark}, we remark that the arguments given in Section~\ref{sec_Fsig_Segre} are also valid for 
any Hibi ring, and give an example.

\section{Some properties on generalized $F$-signatures}
\label{sec_pre_Fsig}

In this section, we show some properties of the generalized $F$-signatures used in this paper. 
Here, we note that for a commutative ring $A$ and an $A$-module $M$, 
$\add_A(M)$ stands for the category consisting of direct summands of finite direct sums of copies of $M$. 

\begin{proposition}
\label{prop_genFsig}
Suppose that $A$ is a strongly $F$-regular local ring. Let $\omega_A$ be the canonical module of $A$, and let $(-)^\vee=\Hom_A(-,\omega_A)$ be the canonical dual. 
\begin{enumerate}[{\rm(a)}]
\setlength{\parskip}{0pt} 
\setlength{\itemsep}{0pt} 
\item We have that $\omega_A\in\add_A(A^{1/p^e})$ for sufficiently large $e\gg0$. 
\item For a finitely generated $A$-module $M$, we see that $M\in\add_A(A^{1/p^e})$ if and only if $M^\vee\in\add_A(A^{1/p^f})$ 
for some $e, f\in\NN$. 
\item Furthermore, we assume that $A$ has FFRT by the FFRT system $\{M_0, M_1, \cdots, M_n\}$. 
Then, we see that $s(M_i,A)=s(M_i^\vee,A)$ for any $i=0, 1, \cdots, n$. 
\end{enumerate}
\end{proposition}

\begin{proof}
(a) If $A$ is Gorenstein, the assertion is trivial. Thus, we assume that $A\not\cong \omega_A$. 
Let $c_e$ be the maximal rank of a free $A$-module appearing in the decomposition of $\omega_A^{1/p^e}$. 
Then, by Theorem~\cite[Theorem~4.11]{Tuc}, we see that $\displaystyle\lim_{e\rightarrow\infty}\frac{c_e}{p^{ed}}=(\rank_A\omega_A)s(A)=s(A)$. 
Since $A$ is strongly $F$-regular, we have that $s(A)>0$, and hence $c_e>0$ for sufficiently large $e\gg0$ in particular. 
Therefore, $A$ appears in $\omega_A^{1/p^e}$ as the direct summand. 
Taking the canonical dual $(-)^\vee$, we have that 
\[
A^{1/p^e}\cong(\omega_A^{1/p^e})^\vee\cong(A^{\oplus c_e}\oplus N_e)^\vee
\cong\omega_A^{\oplus c_e}\oplus N_e^\vee, 
\] 
where $N_e$ has no free direct summands. Thus, we have the conclusion. 

(b) 
First, we assume that $M\in\add_A(A^{1/p^e})$. By applying the canonical dual, we have $M^\vee\in\add_A(\omega_A^{1/p^e})$. 
By (a), we have that $\omega_A\in\add_A(A^{1/p^{e^\prime}})$ for some $e^\prime\in\NN$. 
Therefore, we have $M^\vee\in\add_A(\omega_A^{1/p^e})\subset\add_A(A^{1/p^{e+e^\prime}})$. The converse also holds by the duality.

(c) 
By (b), we see that $M_i, M_i^\vee\in\add_A(A^{1/p^e})$ for sufficiently large $e\gg0$. 
Let $c_e(M_i)$ (resp. $c_e(M_i^\vee)$) be the maximal rank of $M_i$ (resp. $M_i^\vee$) appearing in the decomposition of $A^{1/p^e}$. 
Similarly, let $d_e(M_i)$ (resp. $d_e(M_i^\vee)$) be the maximal rank of $M_i$ (resp. $M_i^\vee$) appearing in the decomposition of $\omega_A^{1/p^e}$.
Since $A^{1/p^e}\cong(\omega_A^{1/p^e})^\vee$, we see that $c_e(M_i)=d_e(M_i^\vee)$ and $c_e(M_i^\vee)=d_e(M_i)$ by the duality of $(-)^\vee$. 
Since $A$ has FFRT, the generalized $F$-signatures of $M_i$ and $M_i^\vee$ exist, 
and we have that 
$$
s(M_i,A)=\lim_{e\rightarrow\infty}\frac{c_e(M_i)}{p^{ed}}=\lim_{e\rightarrow\infty}\frac{d_e(M_i^\vee)}{p^{ed}}=s(M_i^\vee,\omega_A)
=(\rank_A\omega_A)s(M_i^\vee,A)=s(M_i^\vee,A). 
$$ 
Here, the fourth equal sign follows from Remark~\ref{decomp_M}.
\end{proof}

\begin{remark}
We remark that in the proof of Proposition~\ref{prop_genFsig} (c), 
$c_e(M_i)\neq c_e(M_i^\vee)$ in general. In fact, let $A$ be the Veronese subring $\kk[[x,y,z]]^{(2)}$. 
It is known that $A$ is a $3$-dimensional strongly $F$-regular ring that is not Gorenstein. 
Also, $A^{1/p^e}$ is decomposed as $A^{1/p^e}\cong A^{\oplus a_e}\oplus \omega_A^{\oplus b_e}$ 
where $a_e=\displaystyle\frac{p^{3e}+1}{2}$ and $b_e=\displaystyle\frac{p^{3e}-1}{2}$ (see \cite[Example~5.2]{Sei}). 
Thus, we have that $\displaystyle\lim_{e\rightarrow\infty}\frac{a_e}{p^{3e}}=\lim_{e\rightarrow\infty}\frac{b_e}{p^{3e}}=\frac{1}{2}$, but $a_e\neq b_e$. 
\end{remark}

\section{Preliminaries on toric rings}
\label{sec_conic}

In the rest, we turn our attention to toric rings. 
Thus, let $\sfN\cong\ZZ^d$ be a lattice of rank $d$ and let $\sfM\coloneqq\Hom_\ZZ(\sfN, \ZZ)$ be the dual lattice of $\sfN$. 
We set $\sfN_\RR\coloneqq\sfN\otimes_\ZZ\RR$ and $\sfM_\RR\coloneqq\sfM\otimes_\ZZ\RR$. 
We denote an inner product by $\langle\;,\;\rangle:\sfM_\RR\times\sfN_\RR\rightarrow\RR$. 
Let 
$$
\tau\coloneqq\mathrm{Cone}(v_1, \cdots, v_n)=\RR_{\ge 0}v_1+\cdots +\RR_{\ge 0}v_n
\subset\sfN_\RR 
$$
be a strongly convex rational polyhedral cone of dimension $d$ generated by $v_1, \cdots, v_n\in\ZZ^d$ where $d\le n$. 
We assume that $v_1, \cdots, v_n$ are minimal primitive generators of $\tau$. 
For each generator $v_i$, we define a linear form $\sigma_i(-)\coloneqq\langle-, v_i\rangle$, 
and denote $\sigma(-)\coloneqq(\sigma_1(-),\cdots,\sigma_n(-))$. 
We consider the dual cone $\tau^\vee$: 
$$
\tau^\vee\coloneqq\{{\bf x}\in\sfM_\RR \mid \sigma_i({\bf x})\ge0 \text{ for all } i=1,\cdots,n \}. 
$$
Then, $\tau^\vee\cap\sfM$ is a positive normal affine monoid, and hence we define the toric ring 
$$
R\coloneqq \kk[\tau^\vee\cap\sfM]=\kk[t_1^{m_1}\cdots t_d^{m_d}\mid (m_1, \cdots, m_d)\in\tau^\vee\cap\sfM], 
$$
where $\kk$ is an algebraically closed field. 
It is known that $R$ is a $d$-dimensional Cohen-Macaulay (= CM) normal domain. 
Moreover, if ${\rm char}\,\kk=p>0$, then $R$ is strongly $F$-regular (see \cite[Theorem~(5.5)]{HH}). 
In addition, for each $\bfa=(a_1, \cdots, a_n)\in\RR^n$, we set 
$$
\TT(\bfa)\coloneqq\{{\bf x}\in\sfM \mid \text{$\sigma_i({\bf x})\ge a_i$ for all $i=1,\cdots,n$}\}. 
$$
Then, we define the $R$-module $T(\bfa)$ generated by all monomials whose exponent vector is in $\mathbb{T}(\bfa)$. 
By definition, we have $T(\bfa)=T(\ulcorner \bfa\urcorner)$, where $\ulcorner \; \urcorner$ means the round up 
and $\ulcorner \bfa\urcorner=(\ulcorner a_1\urcorner, \cdots, \ulcorner a_n\urcorner)$. 
This $T(\bfa)$ is a divisorial ideal (rank one reflexive module), and any divisorial ideal of $R$ takes this form (see e.g., \cite[Theorem~4.54]{BG2}), 
thus each divisorial ideal is represented by $\bfa\in\ZZ^n$. 
In what follows, we will pay attention to a special class of divisorial ideals called \emph{conic}. 

\begin{definition}[{see e.g., \cite[Section~3]{BG1}}]
\label{def_conic}
We say that a divisorial ideal $T(\bfa)$ is \emph{conic} if there exist $\bfx\in\sfM_\RR$ such that $\bfa=\ulcorner\sigma(\bfx)\urcorner$. 
We denote the set of isomorphism classes of conic divisorial ideals of $R$ by $\calC(R)$. 
\end{definition}

This class of divisorial ideals is so important when we consider the structure of $R^{1/p^e}$. 

\begin{theorem}[{\cite[Proposition~3.6]{BG1},\cite[Proposition~3.2.3]{SmVdB}}]
\label{motivation_thm}
Let $R$ be a toric ring as above, and assume that ${\rm char}\,\kk=p>0$. 
Then, for any $e\in\NN$, the $R$-module $R^{1/p^e}$ is a direct sum of conic divisorial ideals. 
Moreover, all conic divisorial ideals appear in $R^{1/p^e}$ as direct summands for $e\gg 0$. 
\end{theorem}

\begin{remark}
\label{rem_toric_charp}
Although the $R$-module $R^{1/p^e}$ is obtained via the Frobenius morphism, this is also constructed as $\kk[\tau^\vee\cap\frac{1}{p^e}\sfM]$. 
To consider the latter one, we do not need the Frobenius morphism, and hence the assumption ${\rm char}\,\kk=p>0$ is not so important in our situation. 
Nevertheless, we will continue this setting, because our purpose is to determine numerical invariants originated from the positive characteristic framework. 
Thus, in the rest of this paper, we suppose that $\kk$ is an algebraically closed field of ${\rm char}\,\kk=p>0$, in which case $R$ is strongly $F$-regular. 
\end{remark}

Since $R^{1/p^e}$ is a maximal Cohen-Macaulay (= MCM) $R$-module, we have that any conic divisorial ideal is also MCM by Theorem~\ref{motivation_thm}. 
It is known that the number of isomorphism classes of rank one MCM modules is finite (see \cite[Corollary~5.2]{BG1}), 
thus $\calC(R)$ is a finite set. 
As a conclusion, we have that $R$ has FFRT by the FFRT system $\calC(R)$. 

The following is another characterization of conic divisorial ideals. This plays a crucial role for computing the generalized $F$-signatures. 

\begin{lemma}[{see \cite[Corollary 1.2]{Bru}}] 
\label{conic_characterization}
Let the notation be the same as above. 
Let $L_{i,a_i}=\{{\bfx} \in \sfM_\RR \mid a_i-1<\sigma_i({\bfx}) \leq a_i\}$. 
We see that any conic divisorial ideal is isomorphic to $T(a_1,\cdots,a_n)$ for some $(a_1,\cdots,a_n)\in\ZZ^n$ satisfying the following condition; 

\begin{itemize}
\item the cell $\bigcap_{i=1}^n L_{i,a_i}$ is a full-dimensional cell of the decomposition of the semi-open cube $(-1,0]^d$ 
given by hyperplanes $H_{i,m}=\{{\bf x}\in\sfM_\RR \mid \sigma_i({\bf x})=m \}$ for some $m \in\ZZ$ and $i=1,\cdots,n$.
\end{itemize}
\end{lemma}

Thus, observing the decomposition of $(-1,0]^d$ as in Lemma~\ref{conic_characterization}, we can obtain the precise description of conic classes, 
but this computation is sometimes complicated. 
In the next section, we will consider a special class of toric rings called \emph{Hibi rings}, which are defined from a partially ordered set (= poset). 
In that case, we can easily understand the description of conic classes using the associated poset structure (see Theorem~\ref{conic_Hibi}). 

\medskip

Let $R$ be a toric ring as above. As we mentioned, a toric ring is strongly $F$-regular and has FFRT by the FFRT system $\calC(R)$. 
Thus, if we set $\calC(R)=\{M_0:=R, M_1,\cdots, M_n\}$, then for $e\in\NN$, we can write 
$$R^{1/p^e}\cong R^{\oplus c_{0,e}}\oplus M_1^{\oplus c_{1,e}}\oplus\cdots\oplus M_n^{\oplus c_{n,e}}$$
for some $c_{i,e}\ge0$. 
In this case, the generalized $F$-signature $s(M_i,R)=\displaystyle\lim_{e\rightarrow\infty}\frac{c_{i,e}}{p^{ed}}$ exists and is strictly positive as we mentioned in Section~\ref{sec_intro}. 
Furthermore, we can use the combinatorial description of conic divisorial ideals given in Lemma~\ref{conic_characterization} to compute the generalized $F$-signatures. 

\begin{theorem}[{\cite[Section~3]{Bru}, see also \cite{Kor,Sin,WY}}]
\label{Thm_Fsig_volume}
Let $M$ be a conic divisorial ideal of $R$, which certainly appears in $R^{1/p^e}$ as a direct summand for $e\gg 0$ (see Theorem~\ref{motivation_thm}). 
Then, the generalized $F$-signature of $M$ is the sum of volumes of the full-dimensional cell $\bigcap_{i=1}^n L_{i,a_i}$ 
appearing in the decomposition of $(-1,0]^d$ (see Lemma~\ref{conic_characterization}) such that $M\cong T(a_1,\cdots, a_n)$. 
\end{theorem}

\section{Observations on generalized $F$-signatures for Hibi rings}
\label{sec_Fsig_Hibi}

In this section, we study conic divisorial ideals and their generalized $F$-signatures for Hibi rings, which are toric rings arising from partially ordered set (= poset). 
Let $P=\{p_1,\cdots,p_{d-1}\}$ be a finite poset equipped with a partial order $\prec$. 
For $p_i, p_j \in P$, we say that $p_i$ {\em covers} $p_j$ if $p_j \prec p_i$ and there is no $p' \in P$ with $p_i \neq p'$ and $p_j \neq p'$ such that $p_j \prec p' \prec p_i$. 
We say that $P$ is \emph{pure} if all of the maximal chains $p_{i_1} \prec \cdots \prec p_{i_\ell}$ have the same length. 
We denote the unique minimal (resp. maximal) element not belonging to a poset $P$ by $\hat{0}$ (resp. $\hat{1}$), 
and we sometimes use the notation $p_0=\hat{0}$ and $p_d=\hat{1}$. 
We then set $\widehat{P}=P \cup \{\hat{0}, \hat{1}\}$. 
We say that $e=\{p_i,p_j\}$, where $0 \leq i \not= j \leq d$, is an {\em edge} of $\widehat{P}$ if $e$ is an edge of the Hasse diagram of $\widehat{P}$ 
(i.e., $p_i$ covers $p_j$ or $p_j$ covers $p_i$). 
For each edge $e=\{p_i,p_j\}$ of $\widehat{P}$ with $p_i \prec p_j$, 
let $\sigma_e$ be a linear form in $\RR^d$ defined by 
\begin{align*}
\sigma_e({\bf x})\coloneqq
\begin{cases}
x_i-x_j, \;&\text{ if }j \not= d, \\
x_i, \; &\text{ if }j=d 
\end{cases}
\end{align*}
where $\bfx=(x_0,x_1,\cdots,x_{d-1})$. 
Let $\tau_P=\mathrm{Cone}(\sigma_e \mid e \text{ is an edge of }\widehat{P}) \subset \sfN_\RR$. 
The toric ring $\kk[P]\coloneqq\kk[\tau_P^\vee \cap \ZZ^d]$ is called a \emph{Hibi ring} associated with $P$. 
We remark that this definition of a Hibi ring is different from the original one \cite{Hibi}. 
Originally, a Hibi ring is defined via the distributive lattice obtained by considering poset ideals of a poset $P$, 
but our $\kk[P]$ is isomorphic to the original one (see \cite{HHN} for more details). 
Some fundamental properties on Hibi rings were also discussed in \cite{Hibi}. 
For example, it is known that $\dim \kk[P]=|P|+1$, and we have that $\kk[P]$ is Gorenstein if and only if $P$ is pure.

On the other hand, divisorial ideals of $\kk[P]$ form the group called the \emph{class group} $\Cl(\kk[P])$. 
When we consider a divisorial ideal $I$ as the element of $\Cl(\kk[P])$, we denote it by $[I]$. 
By construction, there is a bijection between divisorial ideals with the form $D_i\coloneqq D(0,\cdots,0,\overset{i}{\check{1}},0,\cdots,0)$ and edges of $\widehat{P}$. 
Thus, we denote by $e_i$ the edge of $\widehat{P}$ corresponding to the divisorial ideal $D_i$. 
Let $e_1,\cdots,e_n$ be edges of $\widehat{P}$. 
Then, divisorial ideals $[D_1],\cdots, [D_n]$ generates $\Cl(\kk[P])$. 
To understand the precise generators of $\Cl(\kk[P])$, we prepare some terminologies. 
We say that a sequence $C=(p_{k_1},\cdots,p_{k_m})$ is a {\em cycle} in $\widehat{P}$ if $C$ forms a cycle 
in the Hasse diagram of $\widehat{P}$, that is, $p_{k_i} \not= p_{k_j}$ for $1 \leq i \not= j \leq m$ 
and each $\{p_{k_i},p_{k_{i+1}}\}$ is an edge of $\widehat{P}$ for $1 \leq i \leq m$, where $p_{k_{m+1}}=p_{k_1}$. 
Moreover, we say that a cycle $C$ is a {\em circuit} if $\{p_{k_i},p_{k_j}\}$ is not an edge of $\widehat{P}$ 
for any $1 \leq i,j \leq m$ with $|i-j| \geq 2$. 
Then, we define a {\em spanning tree} as a set of $d$ edges $e_{i_1},\cdots,e_{i_d}$ of $\widehat{P}$ that 
forms a spanning tree of the Hasse diagram of $\widehat{P}$ 
(i.e., any element in $\widehat{P}$ is an endpoint of some edge $e_{i_j}$ and the edges $e_{i_1},\cdots,e_{i_d}$ do not form cycles). 
We remark that a spanning tree is not determined uniquely. 

We now fix the notation used in the rest of this paper. 

\begin{settings}
\label{basic_setting}
Let $P=\{p_1,\cdots,p_{d-1}\}$ be a finite poset, and $e_1,\cdots,e_n$ be the edges of $\widehat{P}=P \cup \{\hat{0}, \hat{1}\}$. 
We fix a spanning tree of $\widehat{P}$ and let $e_1,\cdots,e_d$ be its edges for simplicity. 
Thus, $e_{d+1},\cdots,e_n$ are the remaining edges of $\widehat{P}$. 
Let $\kk[P]$ be the Hibi ring associated with $P$, which is a toric ring with $\dim\kk[P]=d$. 
Here, we assume that $\kk$ is an algebraically closed field of characteristic $p>0$ (see Remark~\ref{rem_toric_charp}). 

Moreover, when we consider the structure of $\kk[P]^{1/p^e}$, we can reduce the case to 
$\widehat{\kk[P]}^{1/p^e}\cong \widehat{\kk[P]}\otimes\kk[P]^{1/p^e}$, where $\widehat{\kk[P]}$ is the $\fkm$-adic completion and 
$\fkm$ is the irrelevant ideal of $\kk[P]$, in which case the Krull-Schmidt condition holds (see Remark~\ref{KS condition}). 
\end{settings}

\begin{theorem}[{see \cite{HHN}}]
Let the notation be the same as in Settings~\ref{basic_setting}. We see that the class group $\Cl(\kk[P])$ is isomorphic to $\ZZ^{n-d}$, 
which is generated by $[D_{d+1}], \cdots, [D_n]$. 
\end{theorem}

In what follows, we will discuss that which elements in $\Cl(\kk[P])\cong \ZZ^{n-d}$ correspond to conic classes. 
For a cycle $C=(p_{k_1},\cdots,p_{k_m})$ in $\widehat{P}$, let 
\begin{align*}
X_C^+&=\{\{p_{k_i},p_{k_{i+1}}\} \mid 1 \leq i \leq m, \;  p_{k_i} \prec p_{k_{i+1}}\}, \\
X_C^-&=\{\{p_{k_i},p_{k_{i+1}}\} \mid 1 \leq i \leq m, \; p_{k_{i+1}} \prec p_{k_i}\}, \\
Y_C^{\pm}&=X_C^{\pm} \cap \{e_1,\cdots,e_d\}, \text{ and }\; 
Z_C^{\pm} = X_C^\pm \cap \{e_{d+1},\cdots,e_n\}, 
\end{align*}
where $p_{k_{m+1}}=p_{k_1}$. 
Let $\calC(P)$ be a convex polytope defined by 
\begin{equation}
\label{ccp}
\calC(P)=\Bigg\{(z_1,\cdots,z_{n-d}) \in \RR^{n-d} \hspace{0.1cm}\Big| \hspace{0.1cm}
-|X_C^-|+1 \leq \sum_{e_{d+\ell} \in Z_C^+}  z_\ell -\sum_{e_{d+\ell'} \in Z_C^-}  z_{\ell'}\leq |X_C^+|-1 \Bigg\},
\end{equation}
where $C=(p_{k_1},\cdots,p_{k_m})$ runs over all circuits in $\widehat{P}$. 
Then, we have a characterization of the conic divisorial ideals of a Hibi ring as follows. 

\begin{theorem}[{see \cite[Theorem~2.4]{HiN}}]
\label{conic_Hibi}
Let the notation be the same as in Settings~\ref{basic_setting}. 
Then, we see that each point in $\calC(P) \cap \ZZ^{n-d}$ one-to-one corresponds to a conic divisorial ideal of $\kk[P]$. 
\end{theorem}

\begin{remark}
The expression of a divisorial ideal as the element of $\Cl(\kk[P])$ depends on a choice of a spanning tree. 
Thus, if we choose another spanning tree, the shape of $\calC(P)$ might be different. 
However, we have a bijection between such different expression via the relations of divisorial ideals. 
\end{remark} 

\begin{example}
\label{ex_typeN}
We consider a finite poset $P=\{p_1,p_2,p_3,p_4\}$ with $p_1\prec p_2$, $p_3\prec p_2$, and $p_3\prec p_4$. 
The following is the Hasse diagram of $\widehat{P}$. 
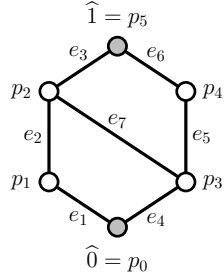
\begin{center}
\newcommand{\edgewidth}{0.07cm} 
\newcommand{\nodewidth}{0.06cm} 
\newcommand{\noderad}{0.2} 
\scalebox{0.6}{
\begin{tikzpicture}
\coordinate (Min) at (0,0); \coordinate (Max) at (0,4);
\coordinate (N1) at (-1.5,1); \coordinate (N2) at (-1.5,3); 
\coordinate (N3) at (1.5,1); \coordinate (N4) at (1.5,3); 

\draw[line width=\edgewidth]  (Min)--(N1) node[black,midway,xshift=-0.1cm,yshift=-0.3cm] {\LARGE $e_1$}; 
\draw[line width=\edgewidth]  (N1)--(N2) node[black,midway,xshift=-0.35cm,yshift=0cm] {\LARGE $e_2$}; 
\draw[line width=\edgewidth]  (Max)--(N2) node[black,midway,xshift=-0.1cm,yshift=0.35cm] {\LARGE $e_3$}; 
\draw[line width=\edgewidth]  (Min)--(N3) node[black,midway,xshift=0.1cm,yshift=-0.33cm] {\LARGE $e_4$}; 
\draw[line width=\edgewidth]  (N3)--(N4) node[black,midway,xshift=0.35cm,yshift=0cm] {\LARGE $e_5$}; 
\draw[line width=\edgewidth]  (Max)--(N4) node[black,midway,xshift=0.1cm,yshift=0.33cm] {\LARGE $e_6$}; 
\draw[line width=\edgewidth]  (N2)--(N3)  node[black,midway,xshift=0cm,yshift=0.35cm] {\LARGE $e_7$}; 

\draw [line width=\nodewidth, fill=lightgray] (Min) circle [radius=\noderad] ; 
\draw [line width=\nodewidth, fill=lightgray] (Max) circle [radius=\noderad] ; 
\draw [line width=\nodewidth, fill=white] (N1) circle [radius=\noderad] ; 
\draw [line width=\nodewidth, fill=white] (N2) circle [radius=\noderad] ;
\draw [line width=\nodewidth, fill=white] (N3) circle [radius=\noderad] ; 
\draw [line width=\nodewidth, fill=white] (N4) circle [radius=\noderad] ;


\node at (-2.1,1) {\LARGE$p_1$} ; \node at (-2.1,3) {\LARGE$p_2$} ; 
\node at (2.1,1) {\LARGE$p_3$} ; \node at (2.1,3) {\LARGE$p_4$} ; 
\node at (0,-0.7) {\LARGE$\widehat{0}=p_0$} ; \node at (0,4.7) {\LARGE$\widehat{1}=p_5$} ; 
\end{tikzpicture} 
}
\end{center}

In this case, the Hibi ring $\kk[P]$ associated with $P$ is defined by the cone in $\RR^5$ generated by 
$$\sigma_{e_1}={}^t(1,-1,0,0,0),\quad \sigma_{e_2}={}^t(0,1,-1,0,0),\quad \sigma_{e_3}={}^t(0,0,1,0,0),\quad \sigma_{e_4}={}^t(1,0,0,-1,0),$$
$$\sigma_{e_5}={}^t(0,0,0,1,-1),\quad \sigma_{e_6}={}^t(0,0,0,0,1),\quad \sigma_{e_7}={}^t(0,0,-1,1,0).$$
In particular, we have that $\dim\kk[P]=5$, and $\kk[P]$ is Gorenstein because $P$ is pure. 
Then, we take a spanning tree $\{e_1, e_2, e_3, e_4, e_5\}$. 
Since the edges $e_6, e_7$ are not contained in this spanning tree, the corresponding divisorial ideals generate the class group; 
$\Cl(R)=\langle [D_6], [D_7] \rangle\cong\ZZ^2$. 
In this case, the Hasse diagram of $\widehat{P}$ contains two circuits $\{p_0,p_1,p_2,p_3\}$ and $\{p_3,p_2,p_5,p_4\}$, and
we have that 
$$
\calC(P)=\{(z_1,z_2)\in\RR^2 \mid -1\le z_2\le1, \, -1\le z_2-z_1\le1\}. 
$$
Thus, the conic class can be represented by nine elements; 
$$
\calC(P)\cap\ZZ^2=\{(0,0), \pm(1,0), \pm(0,1), \pm(1,1), \pm(2,1)\}.
$$
\end{example}

\medskip

The proof of Theorem~\ref{conic_Hibi} (\cite[Theorem~2.4]{HiN}) relies on Lemma~\ref{conic_characterization}, that is, this can be proved by describing the full-dimensional cell decomposition of $(-1,0]^d$ in terms of the associated poset. 
Precisely, we have the description of the full-dimensional cell corresponding to a conic divisorial ideal represented by $\bfc=(c_1,\cdots,c_{n-d}) \in \calC(P) \cap \ZZ^{n-d}$ as follows (we only write the sketch of the argument, see \cite[The proof of Theorem~2.4]{HiN} for more details). 

\begin{observation}
\label{obs_fulldim_cell}
We use the notation in Settings~\ref{basic_setting}. 
We consider the transformation $\phi:\sfM_\RR\cong\RR^d\rightarrow\RR^d$ defined by 
${\bf x}=(x_0,x_1,\cdots,x_{d-1})\mapsto (\sigma_{e_1}(\bfx),\cdots,\sigma_{e_d}(\bfx))\eqqcolon(y_1,\cdots,y_d)$. 
We can see that $\phi$ is unimodular, and hence we may use this new coordinate system for studying conic divisorial ideals of $\kk[P]$. 
Then, for any $j=1,\cdots, {n-d}$, $\sigma_{e_{d+j}}(\bfx)$ can be 
described by the linear combination of $y_1,\cdots,y_d$. 
Precisely, by definition of a spanning tree, for any $j=1, \cdots, n-d$ there exists a unique cycle $C_j=(p_{k_1},\cdots,p_{k_m})$ in $\widehat{P}$ such that 
$\{p_{k_m},p_{k_1}\}=e_{d+j}$, $p_{k_m} \prec p_{k_1}$ and $Z_{C^+_j} \cup Z_{C^-_j} = \{e_{d+j}\}$. 
Using this cycle, we have that 
$$0=\sigma_{e_{d+j}}(\bfx)+\sum_{e_\ell \in Y_{C^+_j}} \sigma_{e_\ell}(\bfx)-\sum_{e_{\ell'} \in Y_{C^-_j}} \sigma_{e_{\ell'}}(\bfx)
=\sigma_{e_{d+j}}(\bfx)+\sum_{e_\ell \in Y_{C^+_j}} y_\ell-\sum_{e_{\ell'} \in Y_{C^-_j}} y_{\ell'}.$$

Thus, let $\widetilde{F}_\bfc$ be the full-dimensional cell consisting of $(y_1,\cdots,y_d)\in\RR^d$ satisfying the following conditions; 
\begin{itemize}
\setlength{\parskip}{0pt} 
\setlength{\itemsep}{3pt}
\item $-1< y_i \leq 0$ \quad for $1 \leq i \leq d$, 
\item $\displaystyle c_j-1 < \sum_{e_{\ell'} \in Y_{C^-_j}} y_{\ell'}-\sum_{e_\ell \in Y_{C^+_j}} y_\ell \leq c_j$ \quad for $1 \leq j \leq n-d$. 
\end{itemize}
Then, by Lemma~\ref{conic_characterization}, we see that  $\widetilde{F}_\bfc$ corresponds to the conic divisorial ideal $T(0,\cdots,0,c_1,\cdots,c_{n-d})$. 
We will denote the closure of $\widetilde{F}_\bfc$ by $F_\bfc$.
\end{observation}

Combining this observation and Theorem~\ref{Thm_Fsig_volume}, 
we have the following corollary, which enables us to compute the generalized $F$-signature for Hibi rings. 

\begin{corollary}
\label{Thm_Fsig_volume_Hibi}
Let the notation be the same as Settings~\ref{basic_setting}. 
Let $M_\bfc$ be the conic divisorial ideal represented by $\bfc=(c_1,\cdots,c_{n-d}) \in \calC(P) \cap \ZZ^{n-d}$, 
that is, $M_\bfc=T(0,\cdots,0,c_1,\cdots,c_{n-d})$. 
Then, the generalized $F$-signature $s(M_\bfc,\kk[P])$ of $M_\bfc$ is equal to the volume of $F_\bfc$. 
\end{corollary}

This would be useful when we compute the generalized $F$-signatures for each concrete Hibi ring. 

\begin{example}
\label{ex_typeN_conic}
Let the notation be the same as Example~\ref{ex_typeN}. 
We consider a conic divisorial ideal $M_\bfc$ corresponding to $\bfc\in\calC(P)\cap\ZZ^2=\{(0,0), \pm(1,0), \pm(0,1), \pm(1,1), \pm(2,1)\}$. 
Computing the volume of $F_\bfc$ (e.g., use the software {\tt Normaliz} \cite{BIS}), we have the generalized $F$-signature $s(M_\bfc,\kk[P])$. 
We will denote $s(M_\bfc)=s(M_\bfc,\kk[P])$ for simplicity, and 
we note that $s(M_\bfc)=s(M_{(-\bfc)})$ by Proposition~\ref{prop_genFsig}. 
\begin{align*}
&s(M_{(0,0)})=s(\kk[P])=\frac{54}{120}, \quad s(M_{(1,0)})=s(M_{(-1,0)})=\frac{13}{120}, \quad s(M_{(1,1)})=s(M_{(-1,-1)})=\frac{13}{120}, \\
&s(M_{(0,1)})=s(M_{(0,-1)})=\frac{6}{120}, \quad s(M_{(2,1)})=s(M_{(-2,-1)})=\frac{1}{120}. 
\end{align*}
Here, we do not reduce fractions, because the denominators of these descriptions are important in the next section. 
\end{example}

\section{Another description of generalized $F$-signatures}
\label{sec_Fsig_Segre}

As we saw in the previous section, we can determine the generalized $F$-signatures of Hibi rings as the volume of a certain full-dimensional cell in semi-open cube $(-1,0]^d$(see Corollary~\ref{Thm_Fsig_volume_Hibi}). 
In this section, we give another description of the generalized $F$-signatures using the elements of symmetric groups. 
That is, we will show that the generalized $F$-signature can be computed by counting certain elements in a symmetric group. 

\subsection{Description of generalized $F$-signatures in terms of symmetric groups}
\label{subsec_Fsig_symmetricgrp}

In this subsection, we explain how to determine the value of the generalized $F$-signatures for Hibi rings using a symmetric group. 
Although the ideas shown in this subsection are valid for any Hibi ring (see Section~\ref{sec_Fsig_remark}), 
we will focus on a special class of Hibi rings for clarifying our arguments. 
Precisely, we consider the Segre product of polynomial rings, which is a Hibi ring as we see below. 

\begin{example}[{see \cite[Example~2.6]{HiN}}]
\label{segre}
Let $R$ be the Segre product $\#_{i=1}^t \kk[x_{i,1},\cdots,x_{i,r_i+1}]$ of polynomial rings, 
which is the ring generated by monomials with the form $x_{1,j_1}x_{2,j_2}\cdots x_{t,j_t}$ where $i=1,\cdots,t$ and $j_i=1,\cdots,r_i+1$. 
We can easily check that $R$ is a Hibi ring associated with the poset whose Hasse diagram is Figure~\ref{segre_poset}. 
In this case, we have that $\dim R=\sum^t_{i=1}r_i+1\eqqcolon d$ and the class group $\Cl(R)$ is isomorphic to $\ZZ^{t-1}$. 
We take a spanning tree $\{e_{i,j} \mid 1\le i\le t, \,\, 1\le j\le r_i\}\cup\{e_{t,0}\}$ (the red edges in Figure~\ref{segre_poset}), 
and consider the circuits $\{p_{i,1},\cdots,p_{i,r_i},p_d,p_{j,r_j},\cdots,p_{j,1},p_0\}$ for $1 \leq i < j \leq t$.
Then, we have the convex polytope $\calC(P)$ given in \eqref{ccp} as follows. 
\begin{align}\label{segre_ineq}
\calC(P)=\{(z_1,\cdots,z_{t-1}) \in \RR^{t-1} \mid &-r_t \leq z_i \leq r_i \text{ for }1 \leq i \leq t-1, \\
& -r_j \leq z_i-z_j \leq r_i \text{ for }1 \leq i < j \leq t-1\}. \nonumber
\end{align}
Let $M_\bfc$ be a divisorial ideal corresponding to $\bfc=(c_1,\cdots,c_{t-1})\in\Cl(R)$. 
Then, we have that $M_\bfc$ is conic if and only if $\bfc\in\calC(P)\cap \Cl(R)$ by Theorem~\ref{conic_Hibi}. 
Moreover, by Observation~\ref{obs_fulldim_cell} 
the closure $F_\bfc$ of the full-dimensional cell $\widetilde{F}_\bfc\subset\RR^d$
corresponding to each point $\bfc=(c_1,\cdots,c_{t-1}) \in \calC(P) \cap \Cl(R)=\calC(P)\cap \ZZ^{t-1}$ is given by 
\begin{align}\label{segre_cell}
\bigg\{(y',y_{i,j})_{\small\substack{1 \leq i \leq t\\ 1 \leq j \leq r_i}} \in \RR^d \mid &-1 \leq y' \leq 0, \; -1 \leq y_{i,j} \leq 0 
\hspace{5pt}\text{ for }1 \leq i \leq t\text{ and } 1\leq j \leq r_i,  \\
&c_i-1 \leq y'+\sum_{j=1}^{r_t}y_{t,j}-\sum_{j'=1}^{r_i}y_{i,j'} \leq c_i \hspace{5pt}\text{ for }1 \leq i \leq t-1\bigg\}.\nonumber
\end{align}

\begin{center}
\begin{figure}[H]
\newcommand{\edgewidth}{0.07cm} 
\newcommand{\nodewidth}{0.06cm} 
\newcommand{\noderad}{0.2} 
\newcommand{\lettersize}{\large} 
{\scalebox{0.6}{
\begin{tikzpicture}
\coordinate (Min) at (4.5,-2.5); \coordinate (Max) at (4.5,8);
\coordinate (N11) at (0,0); \coordinate (N12) at (0,1.5); 
\coordinate (N14) at (0,4); \coordinate (N15) at (0,5.5);
\coordinate (N21) at (3,0); \coordinate (N22) at (3,1.5); 
\coordinate (N24) at (3,4); \coordinate (N25) at (3,5.5);
\coordinate (Ntt1) at (6,0); \coordinate (Ntt2) at (6,1.5); 
\coordinate (Ntt4) at (6,4); \coordinate (Ntt5) at (6,5.5);
\coordinate (Nt1) at (9,0); \coordinate (Nt2) at (9,1.5); 
\coordinate (Nt4) at (9,4); \coordinate (Nt5) at (9,5.5);

\draw[line width=\edgewidth]  (N11)--(Min) node[midway,xshift=0cm,yshift=-0.4cm] {\Large $e_{1,0}$};
\draw[line width=\edgewidth, red]  (N11)--(N12) node[midway,xshift=-0.5cm,red] {\Large $e_{1,1}$}; 
\draw[line width=\edgewidth, red]  (N12)--(0,2) ; \draw[line width=\edgewidth, red]  (0,3.5)--(N14); 
\draw[line width=\edgewidth, loosely dotted, red]  (0,2.3)--(0,3.2);
\draw[line width=\edgewidth, red]  (N14)--(N15) node[midway,xshift=-0.65cm, red] {\Large $e_{1,r_1-1}$}; 
\draw[line width=\edgewidth, red]  (N15)--(Max) node[midway,xshift=-0.3cm,yshift=0.4cm, red] {\Large $e_{1,r_1}$};

\draw[line width=\edgewidth]  (N21)--(Min) node[midway,xshift=-0.6cm,yshift=0.3cm] {\Large $e_{2,0}$};
\draw[line width=\edgewidth, red]  (N21)--(N22) node[midway,xshift=-0.5cm, red] {\Large $e_{2,1}$}; 
\draw[line width=\edgewidth, red]  (N22)--(3,2); \draw[line width=\edgewidth, red]  (3,3.5)--(N24); 
\draw[line width=\edgewidth, loosely dotted, red]  (3,2.3)--(3,3.2);
\draw[line width=\edgewidth, red]  (N24)--(N25) node[midway,xshift=-0.65cm, red] {\Large $e_{2,r_2-1}$}; 
\draw[line width=\edgewidth, red]  (N25)--(Max) node[midway,xshift=-0.9cm,yshift=-0.5cm, red] {\Large $e_{2,r_2}$}; 

\draw[line width=\edgewidth]  (Ntt1)--(Min) node[midway,xshift=0.8cm,yshift=0.3cm] {\Large $e_{t-1,0}$};
\draw[line width=\edgewidth, red]  (Ntt1)--(Ntt2) node[midway,xshift=0.65cm, red] {\Large $e_{t-1,1}$}; 
\draw[line width=\edgewidth, red]  (Ntt2)--(6,2); \draw[line width=\edgewidth, red]  (6,3.5)--(Ntt4); 
\draw[line width=\edgewidth, loosely dotted, red]  (6,2.3)--(6,3.2);
\draw[line width=\edgewidth, red]  (Ntt4)--(Ntt5) node[midway,xshift=0.98cm, red] {\Large $e_{t-1,r_{t-1}-1}$}; 
\draw[line width=\edgewidth, red]  (Ntt5)--(Max) node[midway,xshift=1.1cm,yshift=-0.5cm, red] {\Large $e_{t-1,r_{t-1}}$};

\draw[line width=\edgewidth, red]  (Nt1)--(Min) node[midway,xshift=0cm,yshift=-0.4cm, red] {\Large $e_{t,0}$}; 
\draw[line width=\edgewidth, red]  (Nt1)--(Nt2) node[midway,xshift=0.45cm, red] {\Large $e_{t,1}$}; 
\draw[line width=\edgewidth, red]  (Nt2)--(9,2); \draw[line width=\edgewidth, red]  (9,3.5)--(Nt4); 
\draw[line width=\edgewidth, loosely dotted, red]  (9,2.3)--(9,3.2);
\draw[line width=\edgewidth, red]  (Nt4)--(Nt5) node[midway,xshift=0.65cm, red] {\Large $e_{t,r_t-1}$}; 
\draw[line width=\edgewidth, red]  (Nt5)--(Max) node[midway,xshift=0.2cm,yshift=0.4cm, red] {\Large $e_{t,r_t}$}; 

\draw[line width=\edgewidth, loosely dotted]  (4,2.75)--(5,2.75); 

\draw [line width=\nodewidth, fill=lightgray] (Min) circle [radius=\noderad] ; 
\draw [line width=\nodewidth, fill=lightgray] (Max) circle [radius=\noderad] ; 
\draw [line width=\nodewidth, fill=white] (N11) circle [radius=\noderad] ; \draw [line width=\nodewidth, fill=white] (N12) circle [radius=\noderad] ;
\draw [line width=\nodewidth, fill=white] (N14) circle [radius=\noderad] ;\draw [line width=\nodewidth, fill=white] (N15) circle [radius=\noderad] ;
\draw [line width=\nodewidth, fill=white] (N21) circle [radius=\noderad] ; \draw [line width=\nodewidth, fill=white] (N22) circle [radius=\noderad] ;
\draw [line width=\nodewidth, fill=white] (N24) circle [radius=\noderad] ;\draw [line width=\nodewidth, fill=white] (N25) circle [radius=\noderad] ;
\draw [line width=\nodewidth, fill=white] (Nt1) circle [radius=\noderad] ; \draw [line width=\nodewidth, fill=white] (Nt2) circle [radius=\noderad] ;
\draw [line width=\nodewidth, fill=white] (Nt4) circle [radius=\noderad] ;\draw [line width=\nodewidth, fill=white] (Nt5) circle [radius=\noderad] ;
\draw [line width=\nodewidth, fill=white] (Ntt1) circle [radius=\noderad] ; \draw [line width=\nodewidth, fill=white] (Ntt2) circle [radius=\noderad] ;
\draw [line width=\nodewidth, fill=white] (Ntt4) circle [radius=\noderad] ;\draw [line width=\nodewidth, fill=white] (Ntt5) circle [radius=\noderad] ;

\node at (4.5,-3.1) {\lettersize$p_0=\widehat{0}$}; \node at (4.5,8.6) {\lettersize$p_d=\widehat{1}$}; 
\node at (0.7,0) {\lettersize $p_{1,1}$}; \node at (0.7,1.5) {\lettersize $p_{1,2}$}; 
\node at (0.8,5.5) {\lettersize $p_{1,r_1}$}; 
\node at (3.7,0) {\lettersize $p_{2,1}$}; \node at (3.7,1.5) {\lettersize $p_{2,2}$}; 
\node at (3.8,5.5) {\lettersize $p_{2,r_2}$}; 
\node at (5.2,0) {\lettersize $p_{t-1,1}$}; \node at (5.2,1.5) {\lettersize $p_{t-1,2}$}; 
\node at (5.1,5.5) {\lettersize $p_{t-1,r_{t-1}}$}; 

\node at (8.4,0) {\lettersize $p_{t,1}$}; \node at (8.4,1.5) {\lettersize $p_{t,2}$}; 
\node at (8.35,5.5) {\lettersize $p_{t,r_t}$}; 

\end{tikzpicture}
} }
\caption{The Hasse diagram for the Segre product of polynomial rings}
\label{segre_poset}
\end{figure}
\end{center}
\end{example}

Then we will recall the ideas developed by Lam and Postnikov \cite{LP}. 
Let $z_0=0$ and let 
\begin{align}\label{hypersimplex}
\tilde{\Delta}_{k,d}=\{(z_1,\cdots,z_{d-1}) \in \RR^{d-1} \mid 0 \leq z_i-z_{i-1} \leq 1 \text{ for }i=1,\cdots,d-1, \; k-1 \leq z_{d-1} - z_0 \leq k\}.
\end{align}
This polytope is called the \emph{hypersimplex} (see \cite[the equation (1)]{LP}). 

We say that a lattice polytope $P \subset \RR^d$ is an \emph{alcoved polytope} 
if it is a lattice polytope defined by the hyperplanes all of which are of the form $b_{ij} \leq z_j - z_i \leq c_{ij}$ for $0 \leq i < j \leq d-1$, 
where $b_{ij}$ (resp. $c_{ij}$) might be ``$-\infty$'' (resp. ``$\infty$''), and $b_{ij}=c_{ij}$ (i.e., $z_j-z_i=b_{ij}$) is also fine. 
A typical example of alcoved polytopes is the hypersimplex $\tilde{\Delta}_{k,d}$ above. 

Let $\Sf_d$ be the set of permutations of $d$ integers $\{1,\cdots,d\}$, i.e., $\Sf_d$ is the symmetric group. 
We use one-line notation, i.e., $w_1 w_2 \cdots w_d \in \Sf_d$ stands for the permutation $1 \mapsto w_1, 2 \mapsto w_2, \cdots, d \mapsto w_d$. 
Recall that a \emph{descent} in a permutation $w=w_1\cdots w_d \in \Sf_d$ is an index $i \in \{1,\cdots,d-1\}$ such that $w_i>w_{i+1}$. 
For $w \in \Sf_d$, let $$\des(w)=\big|\{i \in \{1,\cdots,d-1\} \mid \text{$i$ is a descent in $w$}\}\big|.$$ 
We also define $\des(w_iw_{i+1} \cdots w_j)$ for a consecutive subsequence of $w=w_1\cdots w_d \in \Sf_d$ by the similar way, i.e., 
$\des(w_iw_{i+1} \cdots w_j)=\big|\{\ell \in \{i,i+1,\cdots,j-1\} \mid \text{$\ell$ is a descent in $w$}\}\big|$. 

The \emph{Eulerian number} $A_{k,d}$ is defined by $$A_{k,d}=\big|\{w \in \Sf_d : \des(w)=k\}\big|.$$ 
It is known \cite{Sta77} that 
\begin{align}\label{vol_hyp}
\text{the volume of the hypersimplex $\tilde{\Delta}_{k,d}$ is equal to $\displaystyle \frac{A_{k,d-1}}{(d-1)!}$.} 
\end{align}
For more precise information on Eulerian numbers, please refer \cite[Section 1.3]{Sta86}. 

We consider an alcoved polytope $P \subset \RR^d$ which lies within a hypersimplex $\tilde{\Delta}_{k,d+1}$. More precisely, let $P$ be defined as follows: 
\begin{align*}
P=\{(z_1,\cdots,z_d) \in \RR^d \mid \; &0 \leq z_i-z_{i-1} \leq 1 \text{ for }1 \leq i \leq d, \; k-1 \leq z_d-z_0 \leq k, \\
&b_{ij} \leq z_i-z_j \leq c_{ij} \text{ for }1 \leq i < j \leq d\}, 
\end{align*}
where $z_0=0$ and $b_{ij},c_{ij} \in \ZZ$ are parameters given for each pair $(i,j)$ with $0 \leq i< j \leq d$. 
Let $W_P \subset \Sf_d$ be the set of permutations $w=w_1w_2 \cdots w_d \in \Sf_d$ satisfying all of the following conditions: 
\begin{itemize}
\item $\des(w)=k-1$; 
\item $\des(w_iw_{i+1} \cdots w_j) \geq b_{ij}$, and moreover, if $\des(w_iw_{i+1} \cdots w_j) = b_{ij}$ then $w_i<w_j$ holds; 
\item $\des(w_iw_{i+1} \cdots w_j) \leq c_{ij}$, and moreover, if $\des(w_iw_{i+1} \cdots w_j) = c_{ij}$ then $w_i>w_j$ holds, 
\end{itemize}
where we let $w_0=0$. 

\begin{proposition}[{\cite[Proposition 6.1]{LP}}]
\label{key_prop}
The volume of $P$ is equal to $\displaystyle \frac{|W_P|}{d!}$. 
\end{proposition}

We will use this proposition to determine the generalized $F$-signatures for the Hibi ring given in Example~\ref{segre}. 
\begin{theorem}\label{S_n}
Let the notation be the same as Example~\ref{segre}, especially $d=\sum^t_{i=1}r_i+1$. 
Fix $(c_1,\cdots,c_{t-1}) \in \calC(P) \cap \ZZ^{t-1}$. Then, the generalized $F$-signature $s(M_\bfc)$ can be described as follows: 
$$s(M_\bfc)=\frac{1}{d!}\sum_{\alpha_1=0}^{r_1}\cdots\sum_{\alpha_{t-1}=0}^{r_{t-1}} |U(\alpha_1,\cdots,\alpha_{t-1})|,$$ 
where $U(\alpha_1,\cdots,\alpha_{t-1})$ is the set of the permutations $w=w_1\cdots w_d \in \Sf_d$ satisfying the following two conditions: 
\begin{itemize}
\item[(i)] $\displaystyle \des(w_1 \cdots w_{\sum_{q=i+1}^tr_q+1})=r_t-r_i+c_i+\sum_{\ell=i}^{t-1}\alpha_\ell$ for each $i=1,\cdots,t-1$; 
\item[(ii)] $\displaystyle \des(w_{\sum_{q=i}^t r_q} \cdots w_d) \in \left\{\sum_{\ell=1}^{i-1}\alpha_\ell, \sum_{\ell=1}^{i-1}\alpha_\ell+1\right\}$ for each $i=2,\cdots,t$, and 
\begin{itemize}
\item[*] if it is $\displaystyle \sum_{\ell=1}^{i-1}\alpha_\ell$ then $\displaystyle w_{\sum_{q=i}^t r_q} < w_d$ holds, or 
\item[*] if it is $\displaystyle \sum_{\ell=1}^{i-1}\alpha_\ell+1$ then $\displaystyle w_{\sum_{q=i}^t r_q} > w_d$ holds. 
\end{itemize}
\end{itemize}
\end{theorem}
\begin{proof}
By Corollary~\ref{Thm_Fsig_volume_Hibi} together with Observation~\ref{obs_fulldim_cell}, we see that $s(M_\bfc)$ is equal to 
the volume of $F_\bfc$ given in \eqref{segre_cell}. 
We apply the following unimodular transformation: let $X_0=0$ and let 
\begin{align}
\label{eq_Fsig1}
y' \mapsto X_{t,1}-X_0-1 \;\;\text{ and }\;\; y_{i,j} \mapsto X_{i,j+1}-X_{i,j}-1 
\end{align}
for $1 \leq i \leq t$ and $1 \leq j \leq r_i$, where $X_{i,r_i+1}\eqqcolon X_1$ for each $i$. Then the convex polytope $F_\bfc$ becomes as follows: 
\begin{align}\label{cell_part2}
F_\bfc \mapsto \bigg\{(X_{i,j},X_1)_{\small\substack{1 \leq i \leq t\\ 1 \leq j \leq r_i}} \in \RR^d \mid 
&\; 0 \leq X_{t,1}-X_0 \leq 1, \; 0 \leq X_{i,j+1}-X_{i,j} \leq 1 \text{ for }1 \leq i \leq t \text{ and } 1 \leq j \leq r_i \nonumber\\
&r_t-r_i+c_i \leq X_{i,1}-X_0 \leq r_t-r_i+c_i+1 \text{ for }1 \leq i \leq t-1\bigg\}.
\end{align}
It follows from the inequalities in \eqref{cell_part2} that 
$0 \leq \sum_{j=1}^{r_i}( X_{i,j+1}-X_{i,j}) = X_1-X_{i,1} \leq r_i$ and $0 \leq X_1 - X_{i+1,r_{i+1}} \leq 1$ hold, thus we have 
$-r_i \leq X_{i,1}-X_{i+1,r_{i+1}} \leq 1$ holds for each $1 \leq i \leq t-1$. 

Let $\calA=\{(a_1,\cdots,a_{t-1}) \in \ZZ^{t-1} \mid 0 \leq a_i \leq r_i\}$. 
For a given $(\alpha_1,\cdots,\alpha_{t-1}) \in \calA$, 
we consider the region in \eqref{cell_part2} satisfying the inequality 
\begin{align}\label{assume} 
-\alpha_i \leq X_{i,1} - X_{i+1,r_{i+1}} \leq -\alpha_i+1 
\end{align} 
holds for each $1 \leq i \leq t-1$. 
Now, we apply the following parallel translation: 
\begin{align}
\label{eq_Fsig_translation}
X_1 \mapsto z_1-\sum_{\ell=1}^{t-1}\alpha_\ell, \;&\quad
X_{i,j} \mapsto z_{i,j}-\sum_{\ell=i}^{t-1}\alpha_\ell \;\;\text{for each $i, j$},
\;\text{ and we let }\;X_0 = z_0, 
\end{align}
where we define $\sum_{\ell=i}^{t-1}\alpha_\ell=0$ if $i=t$. 
Then \eqref{cell_part2} equipped with \eqref{assume} becomes as follows: 
\begin{align}\label{cell_part3}
F_\bfc \mapsto \bigg\{(z_{i,j},z_1)_{\small\substack{1 \leq i \leq t\\ 1 \leq j \leq r_i}} \in \RR^d \mid 
&\; 0 \leq z_{t,1}-z_0 \leq 1, \; 0 \leq z_{i,j+1}-z_{i,j} \leq 1 \text{ for }1 \leq i \leq t \text{ and } 1 \leq j \leq r_i-1, \nonumber \\
&\; 0 \leq z_{i,1}-z_{i+1,r_{i+1}} \leq 1 \text{ for }1 \leq i \leq t-1, \; 0 \leq z_1 - z_{1,r_1} \leq 1, \\
&\; r_t-r_i+c_i+\sum_{\ell=i}^{t-1}\alpha_\ell \leq z_{i,1}-z_0 \leq r_t-r_i+c_i+\sum_{\ell=i}^{t-1} \alpha_\ell+1 \text{ for }1 \leq i \leq t-1, \nonumber \\
&\; \sum_{\ell=1}^{i-1} \alpha_\ell \leq z_1 - z_{i,r_i} \leq \sum_{\ell=1}^{i-1} \alpha_\ell+1 \text{ for }2 \leq i \leq t\bigg\}. \nonumber 
\end{align}
By renaming the indices $z_0,z_{t,1},z_{t,2},\cdots,z_{t,r_t},z_{t-1,1},\cdots,z_{t-1,r_{t-1}},\cdots,z_{1,r_1},z_1$ 
into $$z_0,z_1,z_2,\cdots,z_{r_t},z_{r_t+1},\cdots,z_{r_t+r_{t-1}},\cdots,z_{r_1+\cdots+r_t},z_d$$
in this order, we see that \eqref{cell_part3} is an alcoved polytope. 
Moreover, 
by the inequalities given in the first two rows of \eqref{cell_part3},
we also see that \eqref{cell_part3} is contained in $\bigcup_{k=1}^{d} \tilde{\Delta}_{k,d+1}$. 
Note that the inequalities 
\begin{equation}\label{ineq_extra}
\begin{split}
&r_t-r_i+c_i+\sum_{\ell=i}^{t-1}\alpha_\ell \leq z_{i,1}-z_0 \leq r_t-r_i+c_i+\sum_{\ell=i}^{t-1} \alpha_\ell+1 \text{ for }1 \leq i \leq t-1 \text{ and }\\
&\sum_{\ell=1}^{i-1}\alpha_\ell \leq z_1-z_{i,r_i}  \leq \sum_{\ell=1}^{i-1}\alpha_\ell+1 \text{ for }2 \leq i \leq t. 
\end{split}
\end{equation}
are extra ones, i.e., those are not contained in the definition of $\tilde{\Delta}_{k,d+1}$. 

For $(\alpha_1,\cdots,\alpha_{t-1}) \in \calA$, let $V(\alpha_1,\cdots,\alpha_{t-1})$ be the set of $w \in \Sf_d$ satisfying the following conditions: 
\begin{itemize}
\item for each $1 \leq i \leq t-1$, we have 
$$r_t-r_i+c_i+\sum_{\ell=i}^{t-1}\alpha_\ell \leq \des(w_0w_1 \cdots w_{\sum_{q=i+1}^tr_q + 1}) \leq r_t-r_i+c_i+\sum_{\ell=i}^{t-1}\alpha_\ell+1,$$ 
and $w_0<w_{\sum_{q=i+1}^tr_q + 1}$ holds if the left-most equality holds or $w_0>w_{\sum_{q=i+1}^tr_q + 1}$ holds if the right-most equality holds; 
\item for each $2 \leq i \leq t$, we have 
$$\sum_{\ell=1}^{i-1}\alpha_\ell \leq \des(w_{\sum_{q=i}^tr_q} \cdots w_{d-1}w_d) \leq \sum_{\ell=1}^{i-1}\alpha_\ell+1,$$ 
and $w_{\sum_{q=i}^tr_q} < w_d$ holds if the left-most equality holds or $w_{\sum_{q=i}^tr_q} > w_d$ holds if the right-most equality holds. 
\end{itemize}
We remark that $z_{i,1}$ is identified with $z_{\sum_{q=i+1}^tr_q + 1}$ and $z_{i,r_i}$ is identified with $z_{\sum_{q=i}^t r_q}$ 
by the above identification of indices. 
Since \eqref{cell_part3} is an alcoved polytope, we have that the volume of \eqref{cell_part3} is equal to 
$\displaystyle \frac{1}{d!}\sum_{(\alpha_1,\cdots,\alpha_{t-1}) \in \calA}|V(\alpha_1,\cdots,\alpha_{t-1})|$ by Proposition~\ref{key_prop}. 
Thus, $s(M_\bfc)$ is equal to this volume. 
Here, we also remark that in the first condition of $V(\alpha_1,\cdots,\alpha_{t-1})$, 
$w_0>w_{\sum_{q=i+1}^tr_q + 1}$ never happens since $w_0=0$. 
Therefore, $V(\alpha_1,\cdots,\alpha_{t-1})$ is equal to $U(\alpha_1,\cdots,\alpha_{t-1})$. 
\end{proof}

\section{Some formulas of generalized $F$-signatures}
\label{sec_Fsig_formula}

In this section, for some Segre products of polynomial rings, we have the formula of the generalized $F$-signatures using 
the idea given in the previous section (see Proposition~\ref{compute_Fsig1} and \ref{compute_Fsig2} below). 

\subsection{The Segre product of the polynomial rings with two variables}
We first consider the Segre product of the polynomial rings with two variables, i.e., 
$$S(t)\coloneqq\kk[x_1,y_1] \# \cdots \# \kk[x_t,y_t].$$ 
Using Example~\ref{segre} for the case $r_1=\cdots=r_t=1$, we see that the conic divisorial ideals of $S(t)$ 
one-to-one correspond to the points contained in 
\begin{align*}\calC(t)\coloneqq\big\{(c_1,\cdots,c_{t-1}) \in \Cl(S(t))\cong\ZZ^{t-1} \mid |c_i | \leq 1 \text{ for }1 \leq i \leq t-1, 
|c_i-c_j| \leq 1 \text{ for } 1 \leq i<j \leq t-1\big\}. \end{align*}
Let $s(c_1,\cdots,c_{t-1})$ denote the generalized $F$-signature of 
the conic divisorial ideal corresponding to $(c_1,\cdots,c_{t-1}) \in \calC(t)$. 

\begin{proposition}
\label{compute_Fsig1}
We have \begin{align*}
&s(\underbrace{0,\cdots,0}_p,\underbrace{1,\cdots,1}_q)=s(\underbrace{0,\cdots,0}_p,\underbrace{-1,\cdots,-1}_q) 
=\frac{1}{\binom{t}{q}(t+1)} \\
&\quad\text{ {\em for any $p \geq 0$ and $q > 0$ with $p+q=t-1$ and any permutation}}, \\
&s(0,\cdots,0)=s(S(t))=\frac{2}{t+1}. 
\end{align*}
\end{proposition}

\begin{proof}
\noindent{\bf (The first step)}: 
First of all, by symmetry, we may only consider conic divisorial ideals represented as 
$(\underbrace{0,\cdots,0}_p,\underbrace{1,\cdots,1}_q)$ or $(\underbrace{0,\cdots,0}_p,\underbrace{-1,\cdots,-1}_q)$ in $\calC(t)$.
Moreover, since $S(t)$ is Gorenstein, we see that $s(\underbrace{0,\cdots,0}_p,\underbrace{1,\cdots,1}_q)=s(\underbrace{0,\cdots,0}_p,\underbrace{-1,\cdots,-1}_q)$ by Proposition~\ref{prop_genFsig} (c). 

\medskip
\noindent{\bf (The second step)}: 
Now, we apply Theorem~\ref{S_n}. 
In our case, we have $r_1=\cdots=r_t=1$, and hence $d=t+1$. 
Thus, for $(c_1,\cdots,c_{t-1})=(\underbrace{0,\cdots,0}_p,\underbrace{1,\cdots,1}_q)$, we obtain that 
\begin{align*}
s(\underbrace{0,\cdots,0}_p,\underbrace{1,\cdots,1}_q)=\frac{1}{(t+1)!}\sum_{(\alpha_1,\cdots,\alpha_{t-1}) \in \{0,1\}^{t-1}}|U(\alpha_1,\cdots,\alpha_{t-1})|, 
\end{align*}
where $U(\alpha_1,\cdots,\alpha_{t-1})$ is the set of permutations $w=w_1 \cdots w_{t+1} \in \Sf_{t+1}$ satisfying 
\begin{itemize}
\item[(i)] $\displaystyle \des(w_1 \cdots w_{t+1-i})=c_i+\sum_{\ell=i}^{t-1} \alpha_\ell$ for each $i=1,\cdots,t-1$; 
\item[(ii)] $\displaystyle \des(w_{t+1-i} \cdots w_{t+1}) \in \left\{\sum_{\ell=1}^{i-1}\alpha_\ell, \sum_{\ell=1}^{i-1}\alpha_\ell+1\right\}$ for each $i=2,\cdots,t$, and 
\begin{itemize}
\item[*] if it is $\displaystyle \sum_{\ell=1}^{i-1}\alpha_\ell$ then $\displaystyle w_{t+1-i} < w_{t+1}$ holds, or 
\item[*] if it is $\displaystyle \sum_{\ell=1}^{i-1}\alpha_\ell+1$ then $\displaystyle w_{t+1-i} > w_{t+1}$ holds. 
\end{itemize}
\end{itemize}
Let 
\begin{align*}
\calU_p=\left\{w=w_1\cdots w_{t+1} \in \Sf_{t+1} \mid \underbrace{w_2,\cdots,w_{t-p}}_q<w_{t+1}=q+1<\underbrace{w_1,w_{t+1-p},w_{t+2-p},\cdots,w_t}_{t-q} \right\}
\end{align*}
for $1 \leq p \leq t-2$, and let 
\begin{align*}
\calU_0=&\left\{w=w_1\cdots w_{t+1} \in \Sf_{t+1} \mid w_2,\cdots,w_{t-1}, w_t<w_{t+1}=t<w_1=t+1 \right\}, \text{ and }\\
\calU_{t-1}=&\{w \in \Sf_{t+1} \mid w_1,\cdots,w_t<w_{t+1}=t+1\} \cup \{w \in \Sf_{t+1} \mid w_1,\cdots,w_t > w_{t+1}=1\}. 
\end{align*}

In what follows, we will prove that 
\begin{align*}
\bigcup_{(\alpha_1,\cdots,\alpha_{t-1}) \in \{0,1\}^{t-1}}U(\alpha_1,\cdots,\alpha_{t-1})=\calU_p \text{ for each }1 \leq p \leq t-1, 
\end{align*}
in which case, we also have that 
\begin{align*}
s(\underbrace{0,\cdots,0}_p,\underbrace{1,\cdots,1}_q)&=\frac{1}{(t+1)!}\sum_{(\alpha_1,\cdots,\alpha_{t-1}) \in \{0,1\}^{t-1}}|U(\alpha_1,\cdots,\alpha_{t-1})|=\frac{1}{(t+1)!}|\calU_p| \\
&=\begin{cases}
\frac{q!(t-q)!}{(t+1)!} \text{ if }p<t-1, \\
\frac{2t!}{(t+1)!} \text{ if }p=t-1, 
\end{cases}
\end{align*}
as required. 

\medskip

We will prove the inclusion $\subset$ in the third step and the other one in the fourth step. 

\medskip

\noindent{\bf (The third step)}: 
Let $w \in \bigcup_{(\alpha_1,\cdots,\alpha_{t-1}) \in \{0,1\}^{t-1}}U(\alpha_1,\cdots,\alpha_{t-1})$. 
Thus, there exists $(\alpha_1,\cdots,\alpha_{t-1})$ satisfying the conditions (i) and (ii) in the second step.
Since $(c_1,\cdots,c_{t-1})=(\underbrace{0,\cdots,0}_p,\underbrace{1,\cdots,1}_q)$, we have 
\begin{align}\label{des(i)}
\des(w_1\cdots w_{t+1-i})=\begin{cases}
\sum_{\ell=i}^{t-1} \alpha_\ell &\text{ for }i=1,\cdots,p, \\
\sum_{\ell=i}^{t-1} \alpha_\ell+1 &\text{ for }i=p+1,\cdots,t-1. 
\end{cases}
\end{align}

First, let $p<t-1$. Then we have $\des(w_1w_2)=\alpha_{t-1}+1$ by \eqref{des(i)}. 
For some $2 \leq i \leq p$, if $\des(w_{t+1-i} \cdots w_{t+1})=\sum_{\ell=1}^{i-1} \alpha_\ell$, then we have 
$$\des(w_1 \cdots w_{t+1})=\des(w_1 \cdots w_{t+1-i})+\des(w_{t+1-i} \cdots w_{t+1})=\sum_{\ell=i}^{t-1}\alpha_\ell+\sum_{\ell=1}^{i-1}\alpha_\ell=\sum_{\ell=1}^{t-1}\alpha_\ell,$$ 
but we also have $$\des(w_1 \cdots w_{t+1})=\des(w_1w_2)+\des(w_2 \cdots w_{t+1}) \geq \alpha_{t-1}+1+\sum_{\ell=1}^{t-2}\alpha_\ell=\sum_{\ell=1}^{t-1}\alpha_\ell+1,$$ 
and this is a contradiction. Hence, $\des(w_{t+1-i} \cdots w_{t+1})=\sum_{\ell=1}^{i-1} \alpha_\ell+1$, and 
\begin{align}\label{eq1}w_{t+1-i} > w_{t+1} \text{ for every }2 \leq i \leq p. \end{align}
Then we also obtain that $\des(w_1\cdots w_{t+1})=\sum_{\ell=1}^{t-1}\alpha_\ell+1$, and hence $w_1>w_{t+1}$. 
On the other hand, when $i=1$, we have that $\des(w_1 \cdots w_{t})=\sum_{\ell=1}^{t-1} \alpha_\ell$ by \eqref{des(i)}, thus we have that $w_t>w_{t+1}$. 

Similarly, for some $p+1 \leq i \leq t-1$, if $\des(w_{t+1-i} \cdots w_{t+1})=\sum_{\ell=1}^{i-1} \alpha_\ell+1$, then we have 
$$\des(w_1\cdots w_{t+1})=\des(w_1 \cdots w_{t+1-i})+\des(w_{t+1-i} \cdots w_{t+1})=\sum_{\ell=i}^{t-1}\alpha_\ell+1+\sum_{\ell=1}^{i-1}\alpha_\ell+1=\sum_{\ell=1}^{t-1}\alpha_\ell+2.$$ 
This contradicts condition (ii) of $U(\alpha_1,\cdots,\alpha_{t-1})$. 
Hence, $\des(w_{t+1-i} \cdots w_{t+1})=\sum_{\ell=1}^{i-1} \alpha_\ell$, and 
\begin{align}\label{eq2}w_{t+1-i} < w_{t+1} \text{ for every }p+1 \leq i \leq t-1. \end{align} 
Therefore, by \eqref{eq1} and \eqref{eq2} together with $w_1>w_{t+1}$ and $w_t>w_{t+1}$, we conclude that $w \in \calU_p$. 

\medskip

Let $p=t-1$ (i.e., $c_1=\cdots=c_{t-1}=0)$. 
Consider the case $w_t<w_{t+1}$. 
Then we have $\des(w_1\cdots w_tw_{t+1})=\des(w_1 \cdots w_t)=\sum_{\ell=1}^{t-1}\alpha_\ell$. 
Hence, we also see that $$\des(w_{t+1-i}\cdots w_{t+1})=\des(w_1 \cdots w_{t+1}) - \des(w_1 \cdots w_{t+1-i})=\sum_{\ell=1}^{t-1}\alpha_\ell-\sum_{\ell=i}^{t-1}\alpha_\ell=\sum_{\ell=1}^{i-1}\alpha_\ell.$$
Therefore, $w_{t+1-i}<w_{t+1}$ for any $2 \leq i \leq t$. 
Similarly, consider the case $w_t>w_{t+1}$. Then we have $\des(w_1\cdots w_tw_{t+1})=\des(w_1 \cdots w_t)+1=\sum_{\ell=1}^{t-1}\alpha_\ell+1$. 
Hence, we also see that $$\des(w_{t+1-i}\cdots w_{t+1})=\des(w_1 \cdots w_{t+1}) - \des(w_1 \cdots w_{t+1-i})=\sum_{\ell=1}^{t-1}\alpha_\ell+1-\sum_{\ell=i}^{t-1}\alpha_\ell=\sum_{\ell=1}^{i-1}\alpha_\ell+1.$$
Therefore, $w_{t+1-i}>w_{t+1}$ for any $2 \leq i \leq t$. 
Thus, $w \in \calU_{t-1}$. 

\medskip

\noindent{\bf (The fourth step)}: 
First, let $p<t-1$. Let $w=w_1\cdots w_{t+1} \in \calU_p$ and let 
\begin{align*}
\alpha_i=\begin{cases}
\des(w_{t-i}w_{t+1-i}) &\text{ if }i \neq p, t-1, \\
1 &\text{ if }i=p, \\
0 &\text{ if }i=t-1. 
\end{cases}
\end{align*}
Then $(\alpha_1,\cdots,\alpha_{t-1})\in \{0,1\}^{t-1}$. 
It follows that $\des(w_1w_2)=1$ and $\des(w_{t-p}w_{t+1-p})=0$. 
Now we can see that $$\des(w_1 \cdots w_{t+1-i})=c_i+\sum_{\ell=i}^{t-1}\alpha_\ell.$$ 
Hence, $w$ satisfies the condition (i). Similarly, we can also verify that $w$ satisfies the condition (ii). 
Therefore, $w \in \bigcup_{(\alpha_1,\cdots,\alpha_{t-1})\in \{0,1\}^{t-1}} U(\alpha_1,\cdots,\alpha_{t-1})$. 

Let $p=t-1$ and let $w=w_1\cdots w_{t+1} \in \calU_{t-1}$. In any cases $w_1,\cdots,w_t<w_{t+1}$ and $w_1,\cdots,w_t>w_{t+1}$, let 
$$\alpha_i=\des(w_{t-i}w_{t+1-i}) \text{ for }1 \leq i \leq t-1. $$
Then we can check that $w$ satisfies the conditions (i) (ii), 
thus $w \in \bigcup_{(\alpha_1,\cdots,\alpha_{t-1})\in \{0,1\}^{t-1}} U(\alpha_1,\cdots,\alpha_{t-1})$. 
\end{proof}

\subsection{The Segre product of two polynomial rings}

We then consider the Segre product of two polynomial rings, 
i.e., $S(r,s) \coloneqq \kk[x_1,\cdots,x_{r+1}] \# \kk[y_1,\cdots,y_{s+1}]$. 
(We are now considering the case with $t=2,r_1=r$ and $r_2=s$ in Example \ref{segre}.) 
By \eqref{segre_ineq}, we see that the conic divisorial ideals of $S(r,s)$ 
one-to-one correspond to the points contained in 
\begin{align*}\calC(r,s)\coloneqq\{c \in \Cl(S(r,s))\cong\ZZ \mid -s \leq c \leq r \}. \end{align*} 
Let $s(c)$ denote the generalized $F$-signature of the conic divisorial ideal corresponding to $c \in \calC(r,s)$. 
Let $d=r+s+1$, which coincides with $\dim S(r,s)=d$.

\begin{proposition}
\label{compute_Fsig2}
We have $s(c)=A_{c+s+1,d}/d!$, where $A_{c+s+1,d}$ is the Eulerian number. 
\end{proposition}
Remark that similar discussions were already given in \cite[Theorem~5.8]{WY}, \cite[Example 7]{Sin}. 

\begin{proof}[Proof of Proposition \ref{compute_Fsig2}]
It suffices to consider the volume of 
$$F_c =\left\{(y_0,y_1,\cdots,y_{r+s}) \in \RR^{r+s+1} \mid -1 \leq y_i \leq 0 \text{ for }0 \leq i \leq r+s, \; 
c-1 \leq \sum_{i=0}^ry_i-\sum_{j=r+1}^{r+s}y_j \leq c \right\}$$ for each $-s \leq c \leq r$. 
We apply the following unimodular transformation: let $z_0=0$ and let 
\begin{align*}
y_0 &\mapsto z_1-z_0-1, \\ 
y_i &\mapsto z_{r+s+1-i}-z_{r+s+2-i} \;\text{ for }\; i=1,\cdots,r, \\
y_{r+j} &\mapsto z_{j+1}-z_j-1 \;\text{ for }\; j=1,\cdots,s. 
\end{align*}
Then we see that 
\begin{align*}
F_c \mapsto \{(z_1,\cdots,z_{r+s+1}) \in \RR^{r+s+1} \mid \; &0 \leq z_i-z_{i-1} \leq 1 \text{ for }i=1,\cdots,r+s+1, \\ 
&c+s \leq z_{r+s+1}-z_0 \leq c+s+1 \}. 
\end{align*}
This is nothing else but the hypersimplex $\tilde{\Delta}_{c+s+1,d+1}$. Therefore, we conclude the assertion by \eqref{vol_hyp}. 
\end{proof}

The unimodular transformation given in the above proof is slightly different from the one given in the proof of Theorem~\ref{S_n}. 

\begin{remark}
There is a close connection between the $F$-signature and the \emph{Hilbert-Kunz multiplicity}, which is also a numerical invariant defined for a commutative algebra in positive characteristic (see e.g., \cite{Hun}). 
The Hilbert-Kunz multiplicity of the Segre products of polynomial rings have been studied in several papers, e.g, \cite{EY,Wat,WY}. 
\end{remark}

\section{Remarks on generalized $F$-signatures for Hibi rings}
\label{sec_Fsig_remark}

In Section~\ref{sec_Fsig_Segre} and \ref{sec_Fsig_formula}, we discussed the generalized $F$-signatures 
for the Segre product of polynomial rings using a symmetric group. 
However, there is no obstruction to apply those arguments to any Hibi ring 
except for the condition imposed to elements in a symmetric group would be untidy. 
We here show the outline of the strategy for general Hibi rings. 

For a Hibi ring $\kk[P]$ of a poset $P$, let us consider $F_{\bf c}$ where ${\bf c} \in \calC(P)\cap\ZZ^{n-d}$. 
For the concrete description of $F_{\bf c}$, see the explanation in Observation~\ref{obs_fulldim_cell}. 
We note that when we fix a spanning tree $T$ of $\widehat{P}$ as a set of edges $\{e_1,\cdots,e_d\}$, 
$\calC(P)$ and $F_{\bf c}$ are determined via $T$, and the variables $y_1,\cdots,y_d$ appearing in the definition of $F_{\bf c}$ 
one-to-one correspond to the edges $e_1,\cdots,e_d$ in $T$. 

We sketch how to transform $F_{\bf c}$ into the polytope whose volume can be computed 
by Proposition~\ref{key_prop} such as done in the proof of Theorem~\ref{S_n}. 

\begin{itemize}
\item[(I)] First, we choose a Hamiltonian path $(e_{h_1},\cdots,e_{h_d})$ of $\widehat{P}$ if possible, 
where we call a path \textit{Hamiltonian} if it contains all vertices of $\widehat{P}$ and paths through each vertex exactly once. 
If $\widehat{P}$ does not contain a Hamiltonian path, then we add some new edges in $\widehat{P}$ and construct a Hamiltonian path. 
\item[(II)] Next, we send each $y_i$ as follows. Assume that $e_i=e_{h_j}$ for some $j$. 
Then, we send $y_i$ to $z_{j}-z_{j-1}-1$ (resp. $z_{j-1}-z_j$) if the direction of $e_{h_j}$ is the same as (resp. different from) 
that of $e_i$. If $e_i$ is not contained in the Hamiltonian path, then we send $y_i$ as follows. 
Let $e_i=\{p,p'\}$ with $p \prec p'$. 
By definition, we can take two edges $e_{h_{a_i}}, e_{h_{b_i}}$ such that 
$p$ (resp. $p'$) is an endpoint of $e_{h_{a_i}}$ (resp. $e_{h_{b_i}}$). 
If $e_{h_{a_i}}$ appears before (resp. after) $e_{h_{b_i}}$ in the Hamiltonian path, 
we send $y_i$ to $z_{b_i}-z_{a_i}-1$ (resp. $z_{b_i}-z_{a_i}$). 
\item[(III)] After the applications of (II), we apply the translation of $z_i$'s to make the resulting polytope 
the one contained in an alcoved polytope. 
\end{itemize} 

We now perform a computation of the generalized $F$-signature for a Hibi ring that is not the Segre product of polynomial rings as follows.

\begin{example}
Let the notation be the same as Example~\ref{ex_typeN}. 
Then, for each $\bfc=(c_1,c_2) \in\calC(P)\cap\ZZ^2=\{(0,0), \pm(1,0), \pm(0,1), \pm(1,1), \pm(2,1)\}$, 
$s(M_\bfc)$ coincides with the volume of the cell 
\begin{align*}
F_\bfc=\{(y_1,\cdots,y_5) \in \RR^5 \mid &-1 \leq y_i \leq 0 \text{ for }i=1,\cdots,5, \\
&c_1-1 \leq (y_1+y_2+y_3)-(y_4+y_5) \leq c_1, \; c_2-1 \leq y_1+y_2-y_4 \leq c_2\}. 
\end{align*}
Here, we see that the edges $e_3,e_2,e_1,e_4,e_5$ form a Hamiltonian path. 
Thus, apply the following unimodular transformation: let $z_0=0$ and let 
$$y_1 \mapsto z_2-z_3, \; y_2 \mapsto z_1-z_2, \; y_3 \mapsto z_0-z_1, \; y_4 \mapsto z_4-z_3-1, \; y_5 \mapsto z_5-z_4-1.$$
Then we see that 
\begin{align*}
F_\bfc \mapsto \{(z_1,\cdots,z_5) \in \RR^5 \mid &0 \leq z_i-z_{i-1} \leq 1 \text{ for }i=1,\cdots,5, \\
&2-c_1 \leq z_5-z_0 \leq 3-c_1, \; 1-c_2 \leq z_4-z_1 \leq 2-c_2\} \subset \tilde{\Delta}_{3-c_1. 6}. 
\end{align*}
Thus, by Proposition~\ref{key_prop}, we obtain that $s(M_\bfc)=\frac{|W|}{5!}$, where $W \subset \Sf_5$ is the set of permutations $w=w_1 \cdots w_5$ satisfying the following: 
\begin{itemize}
\item $\des(w)=2-c_1$; 
\item $\des(w_1w_2w_3w_4) \in \{1-c_2,2-c_2\}$, and if it is $1-c_2$ (resp. $2-c_2$) then $w_1<w_4$ (resp. $w_1>w_4$). 
\end{itemize}

By properly rephrasing these conditions, we see the following: 
\begin{align*}
s(M_{(0,0)})&=\frac{1}{120}\big(|\{w \in \Sf_5 \mid \des(w)=2 \text{ and }w_1<w_4>w_5\} \cup \{w \in \Sf_5 \mid \des(w)=2 \text{ and }w_1>w_4<w_5\}|\big); \\
s(M_{(1,0)})&=\frac{1}{120}|\{w \in \Sf_5 \mid \des(w)=1 \text{ and }w_1<w_4<w_5\}|; \\
s(M_{(0,1)})&=\frac{1}{120}|\{w \in \Sf_5 \mid \des(w_1 \cdots w_4)=1 \text{ and }w_1>w_4>w_5\}|; \\
s(M_{(1,1)})&=\frac{1}{120}\big(|\{w \in \Sf_5 \mid w_1<w_2<w_3<w_4>w_5\}|
+|\{w \in \Sf_5 \mid \des(w_1 \cdots w_4)=1 \text{ and }w_1>w_4<w_5\}|\big); \\
s(M_{(2,1)})&=\frac{1}{120}|\{w \in \Sf_5 \mid w_1<w_2<w_3<w_4<w_5\}|. 
\end{align*}
Those agree with Example~\ref{ex_typeN_conic}. In fact, for example, we can verify that 
\begin{align*}
&\{w \in \Sf_5 \mid \des(w)=1 \text{ and }w_1<w_4<w_5\} \\
=&\{14523, 13524, 13425, 12534, 15234, 12435, 14235, 13245, 25134, 24135, 21345, 23145, 31245\}, 
\end{align*}
and hence $s(M_{(1,0)})=\frac{13}{120}$. 
Moreover, for $\bfc \in \{(1,0),(0,1),(1,1),(2,1)\}$, 
we can observe that the set of permutations corresponding to $s(M_{-\bfc})$ is defined by all ``swapped'' inequalites in that of $s(M_\bfc)$. 
More precisely, for example, we can see that 
$$s(M_{(-1,0)})=\frac{1}{120}|\{w \in \Sf_5 \mid \des(w)=3 \text{ and }w_1>w_4>w_5\}|.$$ 
By the well-known bijection of $\Sf_d$ defined by $w_1 \cdots w_d \mapsto (d+1-w_1) \cdots (d+1-w_d)$, 
we can also get that $s(M_{-\bfc})=s(M_\bfc)$, which was already confirmed by Proposition~\ref{prop_genFsig}(c) in a general situation. 
We note that the corresponding set to $s(M_{(0,0)})$ is closed under this bijection. 
\end{example}

\begin{remark}
In the case of the poset corresponding to the Segre product of polynomial rings (Figure~\ref{segre_poset}), 
we see that there is no Hamiltonian path. Thus, we add new edges $e_i':=\{p_{i,1},p_{i+1,r_{i+1}}\}$ for $i=1,\cdots,t-1$. 
Then, we can choose a Hamiltonian path 
$$e_{t,0},e_{t,1},\cdots,e_{t,r_t-1},e_{t-1}',e_{t-1,1},\cdots,e_{t-1,r_{t-1}-1},e_{t-2}',\cdots,e_{1,r_1-1},e_{1,r_1}.$$ 
Then, the unimodular transformation \eqref{eq_Fsig1} (resp. \eqref{eq_Fsig_translation}) corresponds to 
the application of (II) (resp. (III)). 
\end{remark}

As a consequence of these arguments, we can see that for any $d$-dimensional Hibi ring the values of the generalized $F$-signatures of 
conic divisorial ideals can be described with the form $\Gamma(\mathfrak{S}_d)/d!$, 
where $\Gamma(\mathfrak{S}_d)$ is the number of elements in the symmetric group $\mathfrak{S}_d$ satisfying certain conditions. 

\subsection*{Acknowledgement} 
The authors would like to thank Kei-ichi Watanabe and Ken-ichi Yoshida for valuable discussions. 
The authors also thank the anonymous referee for valuable comments and suggestions.

The first author is supported by JSPS Grant-in-Aid for Young Scientists (B) 17K14177 
and JSPS Grant-in-Aid for Scientific Research (C) 20K03513. 
The second author is supported by World Premier International Research Center Initiative (WPI initiative), MEXT, Japan, 
JSPS Grant-in-Aid for Young Scientists (B) 17K14159, and JSPS Grant-in-Aid for Early-Career Scientists 20K14279. 


\end{document}